\documentclass[twoside,reqno,11pt]{article} %
 \usepackage{hyperref}

\newcommand{\R}{\mathbb{R}}

\usepackage{amsmath}
\usepackage{amssymb}
\usepackage{amsthm}
\usepackage[mathscr]{eucal}
\usepackage{graphicx}

\newtheorem{theorem}{Theorem}[section]
\newtheorem{lemma}{Lemma}[section]
\newtheorem{corollary}{Corollary}[section]
\newtheorem{remark}{Remark}[section]

\begin{document}

  \title{Upper and Lower Estimates\\
  	for the Separation of Solutions\\
  	to Fractional Differential Equations}
 \author{Kai Diethelm\footnote{Faculty of 
 	Applied Natural Sciences and Humanities (FANG), University of Applied Sciences W\"urzburg-Schweinfurt,
	Ignaz-Sch\"on-Str.\ 11, 97421 Schweinfurt, Germany,
	e-mail: kai.diethelm@fhws.de, ORCID: 0000-0002-7276-454X}
	\ and
	Hoang The Tuan\footnote{Institute of Mathematics, Vietnam Academy of Science and Technology,
	18 Hoang Quoc Viet, Cau Giay, Ha Noi, Viet Nam,
	e-mail: httuan@math.ac.vn, ORCID: 0000-0001-7184-1307}
}

\maketitle


\begin{abstract}

	Given a fractional differential equation of order $\alpha \in (0,1]$
	with Caputo derivatives, we investigate in a quantitative sense how the
	associated solutions depend on their respective initial conditions.
	Specifically, we look at two solutions $x_1$ and $x_2$, say,
	of the same differential equation,
	both of which are assumed to be defined on a common interval $[0,T]$,
	and provide upper and lower bounds for the difference $x_1(t) - x_2(t)$
	for all $t \in [0,T]$ that are stronger than the bounds previously
	described in the literature.

\end{abstract}

\medskip

\emph{MSC 2020:} Primary 34A08; Secondary 34A12

\smallskip

\emph{Key words and Phrases}: fractional differential equation; Caputo derivative; initial condition;
	separation of solutions


\section{Introduction and motivation} \label{sec:intro}

\setcounter{equation}{0}\setcounter{theorem}{0}

\subsection{Statement of the problem} 

Initial value problems for fractional differential equations with Caputo derivatives
have proven to be important tools for the mathematical modeling of various phenomena in
science and engineering, see, e.g., 
\cite{BDST2016,BL2019a,BL2019b,Ho2019,Ma2010,Pe2019,Ta2019a,Ta2019b}.
In order to fully understand the
behaviour of such models, it is of interest to precisely describe how their solutions
depend on the initial values. In particular, we shall here look at the following question:
\begin{quote}
	{\sl
	Given two solutions $x_1$ and $x_2$ to the fractional differential equation
	\begin{equation}
		\label{1dimEq}
		{}^{C\!}D_{0+}^{\alpha} x(t) = f(t, x(t)),
	\end{equation}
	where $^{C\!}D_{0+}^{\alpha}$ denotes the Caputo type differential
	operator of order $\alpha \in (0,1]$ with starting point $0$
	\cite[\S 3]{Di2010}, associated to the initial conditions
	$x_1(0) = x_{01}$ and $x_2(0) = x_{02}$, respectively,
	what can be said about the difference $x_1(t) - x_2(t)$ for all $t$
	for which both solutions exist?}
\end{quote}

\vskip 2pt

Our aim is to provide both upper and lower bounds for the difference. A review of the
literature (see Section 
\ref{sec:existing} below) reveals that such bounds exist
in principle, but that in many cases they tend to
be too far away from each other to be of practical use. In other words,
one usually observes that at least one of the two bounds is very weak.
(A concrete example for such a situtation is given in Section 
\ref{sec:example}.)
Therefore, we shall derive tighter inclusions here.

\subsection{Motivation}

From a purely mathematical point of view, such estimates are relevant in
their own right as they allow to draw interesting conclusions about the
behaviour of the solution to the differential equation \eqref{1dimEq}.

In addition, our interest in this question is especially motivated by an application in the
numerical analysis of fractional differential equations that strongly benefits from
tight inclusions. Specifically, one is sometimes interested in \emph{terminal
value problems}, i.e.\ problems of the form
\begin{equation}
	\label{eq:tvp}
	{}^{C\!} D_{0+}^\alpha x(t) = g(t, x(t)),
	\quad
	x(T) = x^*
\end{equation}
with some $T > 0$, and seeks the solution to \eqref{eq:tvp} on the interval $[0,T]$, cf., e.g.,
\cite[pp.\ 107ff.]{Di2010} or \cite{DF2018,FM2011}. For the numerical solution of such problems,
one may apply a so-called \emph{shooting method} \cite{Di2015,DF2012,DU2021,FM2011,FMR2014},
i.e.\ one starts with a first guess $x_{0,1}$ for $x(0)$, (numerically) solves the initial value problem
consisting of the differential equation given in \eqref{eq:tvp} and the initial condition $x(0) = x_{0,1}$,
and in this way obtains a first approximate solution $x^*_1$ for $x(T) = x^*$. One then compares this
approximation $x^*_1$ with the exact value $x^*$, replaces the guess $x_{0,1}$ for the initial value
by a new and improved value $x_{0,2}$ and repeats the process. In order to determine a suitable
choice for $x_{0,2}$, it is useful to have the estimates of the form indicated above because they
describe a connection between $x_{0,2} - x_{0,1}$ one the one hand and $x^* - x^*_1$ on the
other hand, thus telling us which range the new value $x_{01}$ needs to come from in order for the
corresponding initial value problem to have a solution that ``hits'' the required terminal value
as accurately as possible.


\section{Preliminaries}\label{sec.preliminaries}

\setcounter{equation}{0}\setcounter{theorem}{0}

Throughout this paper, we shall use the following conventions.

\vskip 3pt

Let $\alpha\in (0,1]$, $b > 0$, $[0,b]\subset \R$ and $x:[0,b]\rightarrow \R$ be a measurable function
such that $\int_0^b|x(\tau)|\;d\tau<\infty$. The \emph{Riemann--Liouville integral operator
of order $\alpha$} is defined by
\vskip -12pt
\[
	(I_{0+}^{\alpha}x)(t):=\frac{1}{\Gamma(\alpha)}\int_{0}^t(t-\tau)^{\alpha-1}x(\tau)\;d\tau
\]
\vskip -2pt \noindent
for $t \in [0, b]$, where $\Gamma(\cdot)$ is the Gamma function.
The \emph{Riemann--Liouville fractional derivative} $^{RL\!}D_{0+}^\alpha x$ of $x$ on $[0,b]$ is defined by
\[
	({}^{RL\!}D_{0+}^\alpha x) (t):=(DI_{0+}^{1-\alpha}x)(t),\quad\text{for almost all}\;t\in [0,b],
\]
where $D=\frac{d}{dt}$ is the usual derivative. The \emph{Caputo fractional derivative}
of $x$ on $[0,b]$ is defined by
\[
	({}^{C\!}D^\alpha_{0+}x)(t) =  ({}^{RL\!}D_{0+}^\alpha [x-x(0)])(t) \quad \text{for almost all}\;t\in [0,b].
\]

Finally, by $E_\alpha$ we denote the standard one-parameter Mittag-Leffler function, viz.
\vskip -14pt
\[
	E_\alpha(t) = \sum_{k=0}^\infty \frac{t^k}{\Gamma(\alpha k + 1)},
\]
We can cite from \cite[Proposition 3.5]{GKMR2020} the following asysmptotic result that we shall use later:


\begin{lemma}	\label{lemma:ml-asymp}
	Let $\lambda \in \R \setminus \{ 0 \}$. For $t \to \infty$ we have
	\[
		E_\alpha(\lambda t^\alpha)
		= \begin{cases}
				\displaystyle \frac 1 \alpha \exp(\lambda^{1/\alpha} t) + O(t^{-\alpha})
					& \mbox{ for } \lambda > 0, \smallskip \\
				\displaystyle - \lambda \frac{t^{-\alpha}}{\Gamma(1-\alpha)} + O(t^{-2\alpha})
					&\mbox{ for } \lambda < 0,
			\end{cases}
	\]
	so $E_\alpha(\lambda t^\alpha)$ grows exponentially towards $\infty$ if $\lambda > 0$ and
	decays algebraically towards $0$ if $\lambda < 0$.
\end{lemma}

Let now $J = [0,T]$ with some real number $T > 0$ or $J=[0,\infty)$.
As indicated above, we consider the equation \eqref{1dimEq} in this note.
In particular, we shall only discuss the case that $f:J \times \R \to \R$ is a continuous function.
Moreover, we shall generally assume that $f$ satisfies the following
Lipschitz condition on the second variable: there exists a nonnegative continuous
function $L: J \rightarrow \R_+$ such that
\begin{equation}
	\label{eqn.Lip}
	|f(t,x)-f(t,y)| \leq L(t)|x-y| \qquad \hbox{for all } t\in J \hbox{ and all } x,y \in \R.
\end{equation}

\vskip 3pt

The following fundamental results are then known (see \cite{CT2017,Di2010,DST2017}): 


\begin{theorem}
	\label{thm:fundamental}
	Assume that the function $f$ is continuous and satisfies the Lipschitz
	condition \eqref{eqn.Lip}. Moreover, let $x_{10}$ and  $x_{20}$ be two arbitrary real
	numbers with $x_{10} \not= x_{20}$ and consider the two initial value problems
\vskip-11pt
	\begin{subequations}
		\begin{equation}
			\label{eq:ivp1}
			{}^{C\!}D_{0+}^{\alpha} x_1(t) = f(t, x_1(t)), \quad x_1(0) = x_{10}
		\end{equation}
\vskip -4pt \noindent
		and
\vskip -13pt
		\begin{equation}
			\label{eq:ivp2}
			{}^{C\!}D_{0+}^{\alpha} x_2(t) = f(t, x_2(t)), \quad x_2(0) = x_{20},
		\end{equation}
	\end{subequations}
	respectively. Then we have:
	\begin{enumerate}
	\item[(i)] For each of the initial value problems, there exists a unique continuous function
		that solves the problem on the entire interval $J$.
	\item[(ii)] The trajectories of the two solutions do not meet on $J$, i.e., the solutions $x_1(\cdot)$ and
		$x_2(\cdot)$ of \eqref{eq:ivp1} and \eqref{eq:ivp2}, respectively,
		satisfy $x_1(t) \not= x_2(t)$ for all $t \in J$.
	\item[(iii)] In particular, if $x_{10} < x_{20}$ then $x_1(t) < x_2(t)$ for all $t \in J$.
	\end{enumerate}
\end{theorem}  

Note that the two initial value problems \eqref{eq:ivp1} and \eqref{eq:ivp2} differ only in their
initial conditions but contain the same differential equation (which is also the same as the differential
equation given in \eqref{1dimEq}).

\begin{proof}
	Part (i) immediately follows from \cite[Theorem 6.5]{Di2010} (see also \cite[Theorem 2.3]{DST2017}
	or \cite{Ti2012})
	for the case of a finite interval; the extension to the case $J = [0, \infty)$ is immediate,
	cf.\ \cite[Corollary 2.4]{DST2017}. Part (ii)
	has been shown in \cite[Theorem 3.5]{CT2017}, and part (iii) is a direct consequence
	of (ii) in connection with the continuity of $x_1$ and $x_2$ that has been established in part (i).
\end{proof}


\section{Existing results} \label{sec:existing}

\setcounter{equation}{0}\setcounter{theorem}{0}

In connection with the task that we have set, some results have already been derived.
We shall recollect them here in order to demonstrate why it is necessary to find
more accurate bounds. To this end, it is useful to introduce the notation
\vskip -13pt
\begin{equation}
	\label{eq:def-lstar}
	L^*(t) := \max_{\tau \in [0, t]} L(\tau)
\end{equation}
\vskip -2pt \noindent
for $t \in [0, T]$, with $L$ being the function from the Lipschitz condition \eqref{eqn.Lip}.
Using this terminology, we can state the following estimates that are, to the best of our
knowledge, the best currently known bounds for the difference
$|x_1(t) - x_2(t)|$ under the assumptions of Theorem~\ref{thm:fundamental}.


\begin{theorem} \label{thm:lower}
	Under the assumptions of Theorem \ref{thm:fundamental}, the solutions $x_1$ and $x_2$
	of the two initial value problems \eqref{eq:ivp1} and \eqref{eq:ivp2}, respectively, satisfy
	the inequality
\vskip -12pt
	\[
		|x_1(t) - x_2(t)| \ge | x_{10} - x_{20} | \cdot E_\alpha(-L^*(t) t^\alpha)
	\]
\vskip -3pt \noindent
	for all $t \in J$.
\end{theorem}


\begin{theorem} \label{thm:upper}
	Under the assumptions of Theorem \ref{thm:fundamental}, the solutions $x_1$ and $x_2$
	of the two initial value problems \eqref{eq:ivp1} and \eqref{eq:ivp2}, respectively, satisfy
	the inequality
\vskip -12pt
	\[
		|x_1(t) - x_2(t)| \le | x_{10} - x_{20} | \cdot E_\alpha(L^*(t) t^\alpha)
	\]
\vskip -4pt \noindent
	for all $t \in J$.
\end{theorem}

Theorem \ref{thm:lower} is given in \cite[Theorem 4.1]{CT2017}.
Theorem \ref{thm:upper} has been shown in \cite[Theorem 4.3]{CT2017}; slightly weaker forms
can be found in \cite[Theorem 6.20]{Di2010} or \cite[Theorem 4.10]{Ti2012}.


\begin{remark} 
	\label{rmk:multidim}
	It should be noted that, as pointed out in \cite[\S 6]{CT2017}, there is a significant
	difference between Theorems \ref{thm:lower} and \ref{thm:upper} in the sense that
	Theorem~\ref{thm:upper} also holds in the vector valued case, i.e.\ in the case where
	$f: J \times \R^d \to \R^d$ with some $d > 1$, whereas Theorem \ref{thm:lower}
	only holds in the scalar setting.
\end{remark}

To demonstrate the shortcomings of the estimates provided by Theorems
\ref{thm:lower} and \ref{thm:upper}, it suffices to look at the very simple example
of the homogeneous linear differential equation with constant coefficients
\[
	{}^{C\!} D_{0+}^\alpha x(t) = \lambda x(t)
\]
with some real constant $\lambda$, i.e.\ at the case $f(t, x) = \lambda x$. Clearly,
we may choose $J = [0, \infty)$ here. In this
case we can observe the following facts about the initial value problems considered
in the theorems:
\begin{enumerate}
\item The function $L$ is simply given by $L(t) = |\lambda|$; thus, $L^*(t) = |\lambda|$ too.
\item The exact solutions to the initial value problems have the form
	$x_k(t) = x_{k0} E_\alpha (\lambda t^\alpha)$ ($k = 1, 2$).
	Hence,
	\begin{align*}
		|x_1(t) - x_2(t)| & = |x_{10} - x_{20}| \cdot E_\alpha (\lambda t^\alpha)  \\
		&= |x_{10} - x_{20}| \times
			\begin{cases}
				1 & \mbox{ if } \lambda = 0, \\
				E_\alpha(L^*(t) t^\alpha) & \mbox{ if } \lambda > 0, \\
				E_\alpha(-L^*(t) t^\alpha) & \mbox{ if } \lambda < 0. \\					
			\end{cases}
	\end{align*}
\item If $\lambda = 0$ then the upper bound from Theorem \ref{thm:upper}
	coincides with the lower bound from Theorem \ref{thm:lower}, and hence both estimates are sharp.
\item If $\lambda < 0$, the estimate of Theorem \ref{thm:lower} is sharp but, in view of
	 Lemma~\ref{lemma:ml-asymp}, Theorem \ref{thm:upper}
	massively overestimates the difference for large $t$.
\item If $\lambda > 0$, the estimate of Theorem \ref{thm:upper} is sharp but, in view of
	Lemma~\ref{lemma:ml-asymp}, Theorem \ref{thm:lower}
	massively underestimates the difference for large $t$.
\end{enumerate}

This means that we always have an upper bound and a lower bound for $|x_1(t) - x_2(t)|$,
but in all cases except for the trivial case $\lambda = 0$, at least one of these bounds is likely
to be far away from the correct value. Based on this fact, our goal now is to improve those bounds in the sense
that we want to obtain a narrower inclusion, i.e.\ an upper bound and a lower bound that are closer together.
Section \ref{sec:example} below will contain a concrete example that demonstrates a case where the inclusion based
on our new estimates is much tighter than the one based on Theorems \ref{thm:lower} and \ref{thm:upper}.


\section{New and tighter bounds for the difference between solutions} \label{sec:newresults}

\setcounter{equation}{0}\setcounter{theorem}{0}

\subsection{Linear differential equations} 
\label{subs:lin}

We begin our analysis with a look at the special case that the differential
equation under consideration is linear. Much as in \cite{CT2017}, the results for this
special case will later allow us to discuss the general case in Subsection \ref{subs:nonlin}.

Therefore, first consider the equation \eqref{1dimEq} under the assumption that $f(t,x)=a(t)x$ for any
$t\in J$ and $x\in \R$, where $a:J \rightarrow \R$ is continuous. First we formulate
and prove a lower bound for the distance between two solutions.


\begin{theorem}[Convergence rate for solutions of 1-dimensional FDEs]
	\label{thm.convergence.rate}
	Under the conditions of Theorem \ref{thm:fundamental} and the assumption
	\[
		f(t,x)=a(t)x
	\]
	with a continuous function $a : J \to \R$,
	for any $t\in J$ the estimate
	\[
		|x_2(t) - x_1(t)| \geq |x_2(0) - x_1(0)| \cdot E_\alpha\big(a_*(t) t^\alpha\big)
	\]
	holds, where
\vskip-16pt
	\[
		a_*(t) := \min_{\tau\in [0,t]} a(\tau).
	\]
\end{theorem} 


\begin{proof}
	For definiteness we assume $x_2(0) > x_1(0)$. Let $u(t):=x_2(t)-x_1(t)$ for $t\in J$.
	Then by Theorem \ref{thm:fundamental}(iii), we have $u(t) > 0$ for any $t \in J$.
	On the other hand, $u(\cdot)$ is the unique solution to the system
\vskip -10pt
	\begin{subequations}
	\begin{align}
		\label{eq_1} ^{C\!}D_{0+}^{\alpha}u(t)&=a(t)u(t),\;t\in J\setminus\{0\},\\
		\label{eq_2} u(0)&=x_2(0)-x_1(0).
	\end{align}
	\end{subequations}
	For an arbitrary but fixed $t> 0$, we consider the problem
\vskip -10pt
	\begin{subequations}
	\begin{align}
		\label{eq_3} ^{C\!}D_{0+}^{\alpha}v(s) & = a_*(t) v(s),   \;  s \in (0,t],\\
		\label{eq_4} v(0) & = x_2(0) - x_1(0).
	\end{align}
	\end{subequations}
	From \cite[Lemma 3.1]{CT2017} or \cite[Theorem 7.2]{Di2010}, we deduce that this problem
	has the unique solution
	$v(s) = |x_2(0) - x_1(0)| \cdot E_\alpha(a_*(t) s^\alpha)$, $s \in [0,t]$.
	Define $h(s) := u(s) - v(s)$, $s\in [0,t]$. It is easy to see that $h$ is the unique solution of the system
\vskip -12pt
	\begin{subequations}
	\begin{align}
		\label{eq_5} ^{C\!}D_{0+}^{\alpha}h(s)&=a_*(t)h(s)+[a(s)-a_*(t)]u(s),\;t\in (0,t],\\
		\label{eq_6} h(0)&=0.
	\end{align}
	\end{subequations}
\vskip -3pt \noindent
	Notice that for $s\in [0,t]$
\vskip -12pt
	\[
		h(s)=\int_0^s (s-\tau)^{\alpha-1}E_{\alpha,\alpha}(a_*(t)(s-\tau)^\alpha)[a(\tau)-a_*(t)]u(\tau)d\tau,
	\]
\vskip -3pt \noindent
	see also \cite[Lemma 3.1]{CT2017}. Furthermore, $[a(s)-a_*(t)]u(s)\geq 0$ for all $s\in [0,t]$.
	Thus, $h(s)\geq 0$ for all $s\in [0,t]$. In particular, $h(t)\geq 0$ or
	$u(t)\geq v(t) = |x_2(0)-x_1(0)| \cdot E_\alpha(a_*(t)t^\alpha)$. The proof is complete.
\end{proof}  

For the divergence rate and upper bounds for solutions, the following statement is
an easy modification of the well known result.


\begin{theorem}[Divergence rate for solutions of 1-dimensional FDEs]
	\label{thm.divergence.rate}
	Under the assumptions of Theorem \ref{thm.convergence.rate},
	for any $t \in J$ the estimate
\vskip -10pt
	\[
		|x_2(t)-x_1(t)| \leq |x_2(0) - x_1(0)| \cdot E_\alpha(a^*(t) t^\alpha)
	\]
\vskip -3pt \noindent
	holds, where
\vskip -16pt
	\[
		a^*(t)=\max_{s\in [0,t]}a(s).
	\]
\end{theorem}


\begin{proof}
	For definiteness we once again assume $x_2(0) > x_1(0)$ and let $u(t):=x_2(t)-x_1(t)$ for $t\in J$.
	As shown above, $u(t)>0$ for any $t\in J$, and $u(\cdot)$ is the unique solution to the system
	given by eqs.\ \eqref{eq_1} and \eqref{eq_2}.
	For an arbitrary but fixed $t> 0$, this system on the interval $[0,t]$ is rewritten as
\vskip -16pt
	\begin{subequations}
	\begin{align}
		\label{eq_9} ^{C\!}D_{0+}^{\alpha}u(s)&=a^*(t)u(s)+[a(s)-a^*(t)]u(s),\;s\in (0,t],\\
		\label{eq_10} u(0)&=x_2(0)-x_1(0).
	\end{align}
	\end{subequations}
	Thus, due to \cite[Lemma 3.1]{CT2017} we obtain
	\begin{align*}
		u(s)&=|x_2(0)-x_1(0)| \cdot E_\alpha(a^*(t)s^\alpha)\\
		&\hspace{2cm}
			+\int_0^s (s-\tau)^{\alpha-1}E_{\alpha,\alpha}(a^*(t)(s-\tau)^\alpha)[a(\tau)-a^*(t)]u(\tau)d\tau
	\end{align*}
	for $s\in [0,t]$ which together with $[a(s)-a^*(t)]u(s)\leq 0$ for all $s\in [0,t]$ implies that
\vskip -14pt
	\[
		u(s)\leq |x_2(0)-x_1(0)|E_\alpha(a^*(t)s^\alpha)
	\]
	for $s\in [0,t]$. In particular, $u(t)\leq |x_2(0)-x_1(0)| \cdot E_\alpha(a^*(t)t^\alpha)$. The theorem is proved.
\end{proof} 

As an immediate consequence of Theorem \ref{thm.divergence.rate}, we obtain a
stability result for homogeneous linear equations with non-constant coefficients:


\begin{corollary} 	\label{cor:stability-lin}
	Assume the hypotheses of Theorem \ref{thm.convergence.rate} and let $J = [0, \infty)$.
	If $\sup_{t \ge 0} a(t) < 0$ then all solutions $x$ to the equation \eqref{1dimEq} satisfy the
	property $\lim_{t \to \infty} x(t) = 0$. In other words, the differential equation is asymptotically stable.
\end{corollary}


\begin{proof}
	As the differential equation under consideration is linear and homogeneous,
	it is clear that $\tilde x \equiv 0$ is one of its solutions. Moreover, we note that
\vskip-14pt
	\[
		A^* := \sup_{t \ge 0} a^*(t) = \sup_{t \ge 0} a(t).
	\]
\vskip -2pt \noindent
	Thus,	if $x$ is any solution to the differential equation, it
	follows from Theorem \ref{thm.divergence.rate} that
\vskip -12pt
	\begin{align*}
		|x(t)| &= |x(t) - 0| \le | x(0) - 0| \cdot E_\alpha(a^*(t) t^\alpha)
		=  | x(0)| \cdot E_\alpha(a^*(t) t^\alpha) \\
		& \le  | x(0)| \cdot E_\alpha(A^* t^\alpha)
	\end{align*}
\vskip -2pt \noindent
	for all $t \ge 0$ where in the last inequality we have used the well known monotonicity
	of the Mittag-Leffler function $E_\alpha$ \cite[Proposition 3.10]{GKMR2020}.
	Since $A^* < 0$ by assumption, Lemma \ref{lemma:ml-asymp} implies that the
	upper bound tends to $0$ for $t \to \infty$, and our claim follows.
\end{proof} 

\begin{remark}
	Using the arguments developed in \cite{CT2017}, we can see that the
	observations of Remark \ref{rmk:multidim} hold here as well:
	Theorem \ref{thm.divergence.rate} (and hence also Corollary \ref{cor:stability-lin})
	can be generalized to the multidimensional
	setting but Theorem \ref{thm.convergence.rate} cannot.
\end{remark}


\begin{remark}
	The question addressed in Corollary \ref{cor:stability-lin} is closely related to the topic
	discussed (with completely different methods) in \cite{CST2014}.
\end{remark}

\subsection{Nonlinear differential equations} 
\label{subs:nonlin}

Now consider equation \eqref{1dimEq} with $f$ assumed to be continuous on $J\times \R$ and
to satisfy the condition \eqref{eqn.Lip}, so we are in the situation discussed in Theorem
\ref{thm:fundamental}. Further, we assume temporarily that $f(t,0) = 0$ for any $t\in J$.
For each $t \in J$, we define
\vskip -12pt
\begin{equation}
	\label{eq:def-astar}
	a_*(t) := \inf_{s\in [0,t],\;x\in \R\setminus\{0\}}\frac{f(s,x)}{x}
	\quad \text{ and } \quad
	a^*(t) := \sup_{s\in [0,t], \; x \in \R \setminus \{ 0 \}} \frac{f(s,x)}{x}.
\end{equation}
Note that, if the differential equation is linear, i.e.\ if $f(t,x) = a(t) x$, then these
definitions of $a_*$ and $a^*$ coincide with the conventions introduced in Theorems
\ref{thm.convergence.rate} and \ref{thm.divergence.rate}, respectively.

\vskip 3pt

We first state an auxiliary result which asserts that this definition makes sense because the
infimum and the supremum mentioned in \eqref{eq:def-astar} exist.


\begin{lemma}
	\label{lemma:astar}
	Let $f$ satisfy the assumptions mentioned in Theorem \ref{thm:fundamental},
	and assume furthermore that $f(t, 0) = 0$ for all $t \in J$.
	Then, the definitions of the functions $a_*$ and $a^*$ given in \eqref{eq:def-astar}
	are meaningful for all $t \in J$, and the functions $a_*$ and $a^*$ are bounded
	on this interval.
\end{lemma}


\begin{proof}
	By definition, we obtain---in view of the property $f(t, 0) = 0$ and the Lipschitz
	condition \eqref{eqn.Lip}---the estimate
\vskip-8pt
	\[
	-L(t)\leq \frac{f(t,x)}{x}\leq L(t)
	\]
	for any $x\in \R\setminus\{0\}$ and $t\in J$. Thus, for any given time $t\in J$,
\vskip -12pt
	\[
		-\max_{s\in [0,t]}L(s) =\inf_{s\in [0,t]} \left( -L(s) \right) \leq \frac{f(s,x)}{x}
		\quad \forall s\in [0,t],\;x\neq 0.
	\]
	This implies that
\vskip -14pt
	\begin{equation}\label{s1}
	-\max_{s\in [0,t]}L(s)\leq \inf_{s\in [0,t],\;x\neq 0}\frac{f(s,x)}{x}=a_*(t).
	\end{equation}
	On the other hand, we also see that
\vskip -10pt
	\[
	a^*(t)\leq \max_{s\in [0,t]}L(s)
	\]
	for any $t\in J$. This together with \eqref{s1} implies that
	\[
	-\max_{s\in [0,t]}L(s)\leq a_*(t)\leq a^*(t)\leq \max_{s\in [0,t]}L(s)
	\]
	for any $t\in J$. The lemma is proved.
\end{proof}  

\begin{theorem}
	\label{thm:nonlin0}
	Under the assumptions of Lemma \ref{lemma:astar},
	we have:
	\begin{itemize}
	\item[(i)] For $x_0>0$, the solution $\varphi(\cdot,x_0)$ of eq.~\eqref{1dimEq}
		with the condition $x(0)=x_0$ satisfies
\vskip -10pt
		\[
			x_0 E_\alpha(a_*(t)t^\alpha )\leq \varphi(t,x_0 )\leq x_ 0E_\alpha(a^*(t)t^\alpha).
		\]
	\item[(ii)] For $x_0<0$, the solution $\varphi(\cdot,x_0)$ of eq.~\eqref{1dimEq}
		with the condition $x(0)=x_0$ satisfies
\vskip -10pt
		\[
			x_0 E_\alpha(a^*(t)t^\alpha) \leq \varphi(t,x_0 )\leq x_0 E_\alpha(a_*(t)t^\alpha).
		\]
	\end{itemize}
\end{theorem}


\begin{proof}
We only show the proof of the statement (i). The case (ii) is proven similarly. Let $x_0>0$.
By Theorem \ref{thm:fundamental}(iii) and the fact that $f(t,0)=0$ for all $t\in J$, the solution
$\varphi(\cdot,x_0)$ is positive on $J$. For an arbitrary but fixed $t> 0$, we have on the interval $[0,t]$
\vskip -11pt
\[
	^{C\!}D_{0+}^{\alpha}\varphi(s,x_0)
		= a_*(t)\varphi(s,x_0)+\Big(-a_*(t)+\frac{f(s,\varphi(s,x_0))}{\varphi(s,x_0)}\Big)\varphi(s,x_0).
\]
\vskip -4pt \noindent
This implies that
\vskip-14pt
\[
	\varphi(s,x_0)\geq x_0 E_\alpha(a_*(t)s^\alpha),\; s\in [0,t].
\]
In particular, $\varphi(t,x_0)\geq x_0 E_\alpha(a_*(t)t^\alpha)$. On the other hand,
$\varphi(\cdot,x_0)$ is also the unique solution of the equation
\vskip -12pt
\[
	^{C\!}D_{0+}^{\alpha}\varphi(s,x_0)
		= a^*(t)\varphi(s,x_0)+\Big(-a^*(t)+\frac{f(s,\varphi(s,x_0))}{\varphi(s,x_0)}\Big)\varphi(s,x_0),\;s\in [0,t].
\]
\vskip -3pt \noindent
Thus,
\vskip -13pt
\[\varphi(s,x_0)\leq x_0 E_\alpha(a^*(t)s^\alpha),\; s\in [0,t]
\]
and $\varphi(t,x_0)\leq x_0 E_\alpha(a^*(t)t^\alpha)$. The proof is complete.
\end{proof} 

Theorem \ref{thm:nonlin0} has an immediate consequence:


\begin{corollary}
	Assume the hypotheses of Theorem \ref{thm:fundamental}, and furthermore let $f(t,0) = 0$
	for all $t \in J$.
	\begin{enumerate}
	\item[(i)] For $0 < x_{10} < x_{20}$, we have for all $t \in J$ that
		\begin{eqnarray*}
			\lefteqn{ x_{20} E_\alpha (a_*(t) t^\alpha) - x_{10} E_\alpha (a^*(t) t^\alpha) } \\
			& \le  & x_2(t) - x_1(t) \\
			& \le & x_{20} E_\alpha (a^*(t) t^\alpha) - x_{10} E_\alpha (a_*(t) t^\alpha).
		\end{eqnarray*}
	\item[(ii)] For $x_{10} < 0 < x_{20}$, we have for all $t \in J$ that
		\[
			(x_{20} - x_{10}) E_\alpha (a_*(t) t^\alpha)
			\le   x_2(t) - x_1(t)
			\le (x_{20} - x_10) E_\alpha (a^*(t) t^\alpha).
		\]
	\item[(ii)] For $x_{10} < x_{20} < 0$, we have for all $t \in J$ that
		\begin{eqnarray*}
			\lefteqn{ x_{20} E_\alpha (a^*(t) t^\alpha) - x_{10} E_\alpha (a_*(t) t^\alpha) } \\
			& \le  & x_2(t) - x_1(t) \\
			& \le & x_{20} E_\alpha (a_*(t) t^\alpha) - x_{10} E_\alpha (a^*(t) t^\alpha).
		\end{eqnarray*}	
	\end{enumerate}
\end{corollary} 

From this result, we can also deduce an analog of Corollary \ref{cor:stability-lin},
i.e.\ a sufficient criterion for asymptotic stability, for the nonlinear case.


\begin{corollary}
	Assume the hypotheses of Theorem \ref{thm:fundamental}, and furthermore let $J = [0, \infty)$ and
	$f(t,0) = 0$ for all $t \in J$. Moreover, let $\sup_{t \ge 0} a^*(t) < 0$. Then, all solutions $x$
	of the differential equation \eqref{1dimEq} satisfy $\lim_{t \to \infty} x(t) = 0$.
\end{corollary}

The proof is an immediate generalization of the proof of Corollary \ref{cor:stability-lin}.
We omit the details.

\vskip 3pt

We now give up the requirement that $f(t, 0) = 0$. To this end, we essentially follow the standard
procedure in the analysis of stability properties of differential equations; cf., e.g., \cite[Remark 7.4]{Di2010}.


\begin{theorem}
	\label{thm:nonlin-general}
	Assume the hypotheses of Theorem \ref{thm:fundamental}, and let $x_{10} < x_{20}$.
	Then, for any $t \in J$ we have
	\[
		(x_2(0) - x_1(0)) E_\alpha (\tilde a_*(t) t^\alpha)
			\le x_2(t) - x_1(t)
			\le (x_2(0) - x_1(0)) E_\alpha (\tilde a^*(t) t^\alpha)
	\]
	where
\vskip -12pt
	\begin{subequations}
	\begin{equation}
		\label{eq:tilde_a_star_inf}
		\tilde a_*(t) = \inf_{s \in [0, t], \, x \ne 0} \frac{f(s , x + x_1(s)) - f(s, x_1(s))}{x}
	\end{equation}
\vskip -2pt \noindent
	and
\vskip -12pt
	\begin{equation}
		\label{eq:tilde_a_star_sup}
		\tilde a^*(t) = \sup_{s \in [0, t], \, x \ne 0} \frac{f(s , x + x_1(s)) - f(s, x_1(s))}{x}.		
	\end{equation}
	\end{subequations}
\end{theorem} 

\begin{proof}
	First, we note that, in view of Theorem \ref{thm:fundamental}(iii), we have
	$x_1(t) < x_2(t)$ for all $t \in J$. Then we define the function
	\[
		\tilde f(t, x) = f(t, x + x_1(t)) - f(t, x_1(t))
	\]
\vskip -3pt \noindent
	and notice that
	\[
		^{C\!}D_{0+}^{\alpha} (x_2 - x_1)(t)
			= {}^{C\!}D_{0+}^{\alpha} x_2(t) - {}^{C\!}D_{0+}^{\alpha} x_1 (t)
			= f(t, x_2(t)) - f(t, x_1(t)),
	\]
	so that the function $\tilde x := x_2 - x_1$ satisfies the differential equation
	\[
		^{C\!}D_{0+}^{\alpha} \tilde x(t) = \tilde f(t, \tilde x(t))
	\]
	and the initial condition $\tilde x(0) = x_2(0) - x_1(0) > 0$.
	Moreover, $\tilde f(t,0) = 0$ for all $t$, and $\tilde f$ satisfies the Lipschitz condition \eqref{eqn.Lip} with the same
	Lipschitz bound $L(t)$ as $f$ itself. This implies that the quantities $\tilde a_*(t)$
	and $\tilde a^*(t)$ exist and are finite. Furthermore, we may apply Theorem \ref{thm:nonlin0}(i) to the function
	$\tilde x$ and derive the claim.
\end{proof} 

Note that Theorem \ref{thm:nonlin-general} is the only result in Section \ref{sec:newresults} whose application in
practice requires the knowledge of an exact solution to the given differential equation. All other results are solely based
on information about the given function $f$ on the right-hand side of the differential equation.


\section{An application example} \label{sec:example}

\setcounter{equation}{0}\setcounter{theorem}{0}

As an application example, we consider the linear differential equation
\begin{equation}
	\label{eq:ex1}
	{}^{C\!} D^\alpha_{0+} x(t) = - \frac 1 2 ( 1 + 4 t +3 \cos 4 t) x(t)
\end{equation}
for $t \in [0,\infty)$.
In the notation of Subsection \ref{subs:lin}, we have
\[
	a(t) =  - \frac 1 2 ( 1 + 4 t +3 \cos 4 t).
\]
The function $a^*$ defined in Theorem \ref{thm.divergence.rate} satisfies
\[
	A^* := \sup_{t \ge 0} a^*(t) = \sup_{t \ge 0} a(t) < 0;
\]
therefore, by Corollary \ref{cor:stability-lin}, the equation is
asymptotically stable and hence a prototype of a class of problems
that is particularly relevant in practice.

\vskip 3pt

For the purposes of concrete experiments, we restrict our attention
to the interval $J = [0, T]$ with $T = 6$.
We have plotted the function $a$ and the
associated funtions $a_*$ and $a^*$ on this interval in Fig.\ \ref{fig:plot-a}.
In particular, we can compute (and see in the figure) that
\[
	a_*(T) = a_*(6) = - \frac{25}2 - \frac 3 2 \cos(24) \approx - 13.136
\]
\vskip -6pt \noindent
and
\vskip -14pt
\[
	A^* = a(\frac 1 4 (\pi - \arcsin \frac 1 3))
		= \frac 1 2 ( \sqrt 8 - 1 - \pi + \arcsin \frac 1 3 ) \approx  -0.4867.
\]

\begin{figure}
	\centering
	\includegraphics[width=0.8\textwidth]{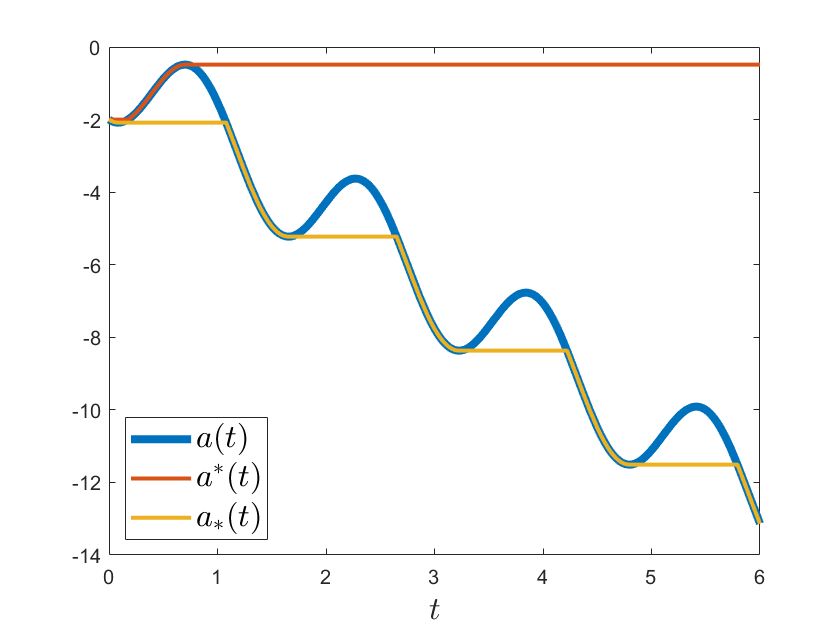} 
	\caption{\label{fig:plot-a} The functions $a(t)$, $a_*(t)$ and $a^*(t)$ for the
		example from eq.\ \eqref{eq:ex1}.}
\end{figure} 

To demonstrate the effectiveness of our new estimates, we choose
$\alpha = 0.65$ and consider two
solutions $x_1$ and $x_2$ to the differential equation \eqref{eq:ex1}
subject to the initial conditions $x_1(0) = 1$ and $x_2(0) = 2$, respectively.
Since exact solutions for these two initial value problems are not available,
we have reverted to numerical solutions instead. To this end, we have used
Garrappa's fast implementation of the fractional trapezoidal method \cite{Ga2015}
that is based on the ideas of Lubich et al.\ \cite{HLS1985,Lu1986}. We have used the
step size $h = 10^{-5}$ which, in combination with the well known
stability properties of this numerical method, allows us to reasonably believe that
the numerical solution is very close to the exact solution.
Figure \ref{fig:ex1} shows the graphs of the two solutions.

\begin{figure} 
	\centering
	\includegraphics[width=0.8\textwidth]{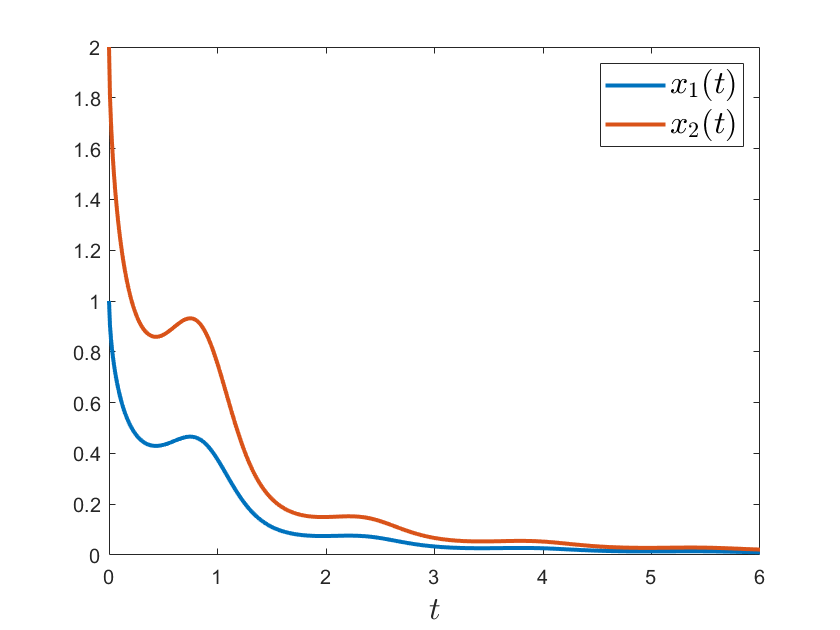} 
	\caption{\label{fig:ex1} The graphs of the solutions $x_1$ and $x_2$ to the
		example from eq.\ \eqref{eq:ex1}.}
\end{figure} 

The essential observation can be read off from Fig.\ \ref{fig:comparison}. Since $a(t) < 0$ for
all $t$ in this example, it can be seen that the function $L$ from the Lip\-schitz condition
of the differential equation's right-hand side is just $L(t) = |a(t)| = -a(t)$, and
hence the function $L^*$ from Theorems \ref{thm:lower} and \ref{thm:upper}
is simply $L^*(t) = - a_*(t)$. Therefore, the old lower bound of Theorem \ref{thm:lower}
is identical to the new bound of Theorem \ref{thm.convergence.rate}. The fact that we
have been unable to improve this bound in the example reflects the fact that the old bound
is already very close to the correct value of the difference between the two functions.
For the two upper bounds, however, we obtain a completely different picture.
While the old bound from Theorem \ref{thm:upper} vastly overestimates the
true value of the difference (note the logarithmic scale on the vertical axis
of Fig.\ \ref{fig:comparison}), the new bound is very much closer. In particular,
our new bound---like the true difference---tends to 0 as $t \to \infty$ whereas
the previously known bound tends to $\infty$.

\begin{figure} 
	\centering
	\includegraphics[width=0.8\textwidth]{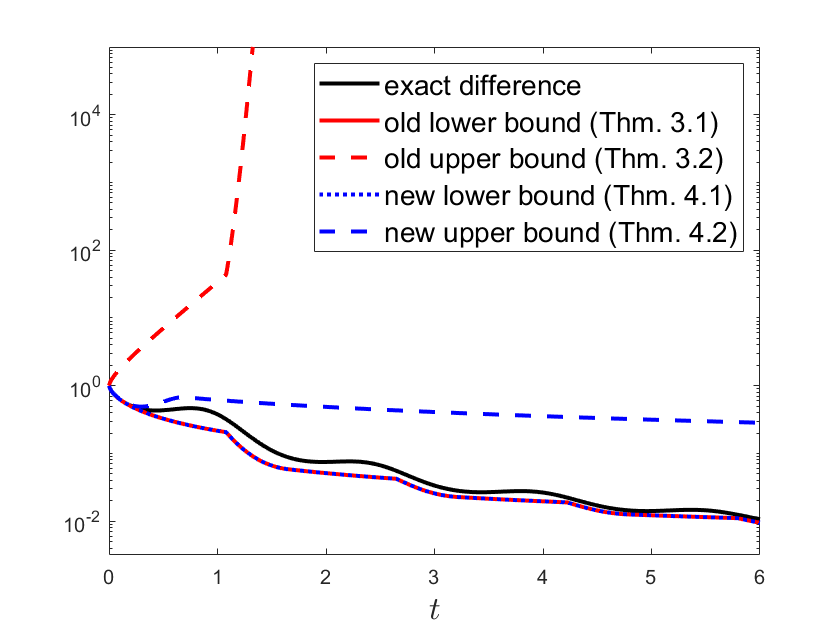}  
	\caption{\label{fig:comparison} Comparison of true differences between the solutions $x_1$ and
		$x_2$ (black) with the associated new upper and lower estimates derived in
		Theorems \ref{thm.convergence.rate} and \ref{thm.divergence.rate}, respectively
		(blue) and the corresponding estimates obtained by previously known methods
		listed in Theorems \ref{thm:lower} and \ref{thm:upper} (red).}
\end{figure}

\section*{Acknowledgement}

Hoang The Tuan was funded by Vingroup JSC and supported by the
Postdoctoral Scholarship Programme of the Vingroup Innovation Foundation (VINIF),
Vingroup Big Data Institute (VinBigdata), under the code\linebreak  VINIF.2021.STS.17.



\end{document}